\documentclass{elsarticle}

\usepackage{amsmath,amsthm,amssymb}
\usepackage{lineno}
\usepackage{latexsym}
\usepackage{verbatim}
\usepackage{enumerate}
\usepackage{color}

\newcommand{\Hall}{\mathrm{Hall}}
\newcommand{\AG}{\mathrm{AG}}
\newcommand{\PG}{\mathrm{PG}}
\newcommand{\GF}{\mathrm{GF}}
\newcommand{\PGL}{\mathrm{PGL}}

\newtheorem{theorem}{Theorem}[section]
\newtheorem{corollary}[theorem]{Corollary}
\newtheorem{proposition}[theorem]{Proposition}
\newtheorem{lemma}[theorem]{Lemma}

\newtheorem{remark}[theorem]{Remark}

\newtheorem{question}[theorem]{Question}

\journal{European Journal of Combinatorics}

\begin{document}

\begin{frontmatter}

\title{Inherited conics in Hall planes}

\author[eindhoven]{Aart Blokhuis} \ead{a.blokhuis@tue.nl}

\author[upiam,famnit]{Istv\'an Kov\'acs\fnref{ikthanks}}
\ead{istvan.kovacs@upr.si}

\author[bme,bolyai]{G\'abor P. Nagy\fnref{gnthanks,twothanks}}
\ead{nagyg@math.u-szeged.hu}

\author[elte,gac,famnit]{Tam\'as Sz\H onyi\fnref{tszthanks,twothanks}}
\ead{szonyi@cs.elte.hu}

\address[eindhoven]{Department of Mathematics and Computer Science, Eindhoven
University of Technology, Den Dolech 2, Eindhoven, The Netherlands}

\address[upiam]{UP IAM, University of Primorska, Muzejski trg 2, 6000, Koper,
Slovenia} \address[famnit]{UP FAMNIT, University of Primorska, Glagolja\v{s}ka
8, 6000, Koper, Slovenia}

\address[bme]{Department of Algebra, Budapest University of Technology, H-1111,
Budapest, Egri J\'ozsef utca 1, Hungary} \address[bolyai]{Bolyai Institute,
University of Szeged, H-6720 Szeged, Aradi v\'ertan\'uk tere 1, Hungary}

\address[elte]{Department of Computer Science, E\"otv\"os Lor\'and University,
H-1117 Budapest, P\'azm\'any P\'eter s\'et\'any 1/C, Hungary}

\address[gac]{MTA-ELTE Geometric and Algebraic Combinatorics Research Group,
H-1117 Budapest, P\'azm\'any P\'eter s\'et\'any 1/C, Hungary}

\fntext[ikthanks]{Partially supported by the Slovenian Research Agency (research
program P1-0285 and research projects N1-0038, N1-0062, J1-6720 and J1-7051).}

\fntext[gnthanks]{Partially supported by OTKA grant 119687.}


\fntext[twothanks]{Partially supported by the
OTKA-ARRS Slovenian-Hungarian Joint Research Project, grant no. NN 114614 (in
Hungary) and N1-0032 (in Slovenia).}

\begin{abstract}
The existence of ovals and hyperovals is an old question in the
theory of non-Desarguesian planes. The aim of this paper is to
describe when a conic of ${\rm PG}(2,q)$ remains an arc in the
Hall plane obtained by derivation. Some combinatorial properties
of the inherited conics are obtained also in those cases when
it is not an arc. The key ingredient of the proof is an old lemma
by Segre-Korchm\'aros on Desargues configurations {with perspective
triangles} inscribed in a
conic.
\end{abstract}

\end{frontmatter}


\section{Introduction}

An \emph{arc} in a projective plane is a set of points no three of which are
collinear. An old theorem of Bose says that an arc can have at most $q+2$ points
if $q$ is even, and at most $q+1$ points if $q$ is odd. An arc having $k$ points
is also called a $k$-arc. A $k$-arc is said to be \emph{complete} if it is not
contained in a $(k+1)$-arc. $(q+1)$-arcs are called \emph{ovals}, $(q+2)$-arcs
are called \emph{hyperovals}. It is also known that ovals in planes of even
order are contained in a (unique) hyperoval. Arcs and ovals are among the most
studied objects in finite geometry. A motivating question was the existence of
ovals in any (not necessarily Desarguesian) plane. Several results are known for
arcs in the Desarguesian plane ${\rm PG}(2,q)$, here we just mention some of
them.

\begin{theorem}[Segre, \cite{S}] 
If $K$ is a complete $k$-arc in ${\rm PG}(2,q)$, then $k=q+2$ or $k\le q-\sqrt
q+1$ if $q$ is even, and $k=q+1$, in which case $K$ is a conic, or $k\le
q-\frac14\sqrt q+\frac74$ if $q$ is odd.
\end{theorem}

For a survey of results on arcs and blocking sets we refer to the book by
Hirschfeld \cite{Hir}. Relatively few results are known for (complete) arcs in
non-Desarguesian planes. In particular, no embeddability results similar to
Segre's theorems are known. Instead of giving a full list of the results we just
refer to an old survey paper by the fourth author \cite{SzTBari} and pick some
characteristic results about arcs. Of course, the focus was on non-Desarguesian
planes which are close to Galois planes. This means that most results are about
arcs of Hall planes, Andr\'e planes and their duals (Moulton planes). In the
early years, researchers wanted to find ovals and hyperovals in non-Desarguesian
planes. There are such examples by Rosati, Bartocci, Korchm\'aros \cite[Theorem
3.1]{SzTBari}. An early important result about complete arcs is due to
Menichetti: there are complete $q$-arcs in Hall planes of even order ($\ge 16$)
\cite{Men}. A similar but easier result is due to Sz\H onyi: there are complete
$(q-1)$-arcs in Hall planes of odd order \cite[Theorem 4.6]{SzTBari}. A natural
idea is to start with an oval (or a conic) of the Desarguesian plane and study
the combinatorial properties of these sets in the non-Desarguesian plane
(obtained from the Desarguesian one by replacing some of the lines).

In this paper, we shall systematically study inherited conics in Hall planes. In
the next section some fundamental results used in the proofs are collected. Then
we discuss old and new results about different types of conics: parabolas,
hyperbolas, and ellipses and decide whether they yield inherited arcs or not.
Some cases were completely known before, some were not. The precise results are
stated in the corresponding sections.

We should remark that Barwick and Marshall \cite{BarMar} found a necessary and
sufficient condition in terms of the equation of the conic guaranteeing that it
remains an arc in the Hall plane. The disadvantage of the result is that the
condition is not easy to check explicitly.

Throughout the paper {\em conic} will stand for {\em irreducible conic}.

\section{The Hall plane}\label{hall}

In this section, the Hall planes are described briefly by using derivation and
also by giving the lines explicitly.

Consider the Desarguesian projective plane ${\rm PG}(2,q^2)$, let $\ell$ be a
line and let $D$ be a Baer subline of $\ell$. So $D\cong {\rm PG}(1,q)\subset
{\rm PG}(1,q^2)=\ell$. We call $\ell{=\ell_\infty}$ the line at infinity. The
points of the affine Hall plane $\Hall(q^2)$ are the points of ${\rm
PG}(2,q^2)\setminus {\ell_\infty}$. Lines whose infinite point does not belong
to $D$ remain the same (`old lines'). Instead of lines intersecting
${\ell_\infty}$ in a point of $D$ we consider all Baer subplanes containing $D$.
The affine part of these Baer subplanes are the `new lines'. It is not difficult
to show that this incidence structure is an affine plane (and the translations
of the Desarguesian affine plane are translations in the Hall plane). The
projective Hall plane is the projective closure of this affine plane.

For the sake of completeness, we describe the affine Hall plane $\Hall(q^2)$
explicitly. Points are the pairs $(x,y)$, where $x,y\in \GF(q^2)$. Old lines
have equation $Y=mX+b$, where $m\notin \GF(q)$. New lines are $\{(a+\lambda
u,b+\lambda v) : u,v\in \GF(q)\}$, where $a,b,\lambda$ are fixed elements of
$\GF(q^2)$. Note that the same Baer subplane is obtained for several
$a,b,\lambda$. In this case we have the {\em standard derivation set}, `the
usual ${\rm PG}(1,q)$':
\[
D=\{(m)\mid m\in \GF(q)\cup\{\infty\}\}=\{(x:y:0)\mid x,y\in\GF(q)\}.
\]

\section{Useful facts about conics}

Let us begin with the following {result} by
Segre and Korchm\'aros {\cite[page 617]{SK}} which plays
a crucial role in our proof.

\begin{theorem}[Segre-Korchm\'aros]\label{sk1}
\begin{enumerate}[(a)]
\item Let $K$ be a conic of ${\rm PG}(2,q)$, $q$ even, and $r$ be a line which
is not a tangent of $K$. For every triple $\{P_1,P_2,P_3\}\subset r\setminus K$
there exists one and only one triangle $\{A_1,A_2,A_3\}$ inscribed in
$K\setminus r$ such that $A_iA_j\cap r=P_k$, where $i,j,k$ is a permutation of
$1,2,3$.
\item Let $K$ be a conic of ${\rm PG}(2,q)$, $q$ odd, and $r$ be a line which is
not a tangent of $K$. For every triple $\{P_1,P_2,P_3\}\subset r\setminus K$
there exist at most two triangles $\{A_1,A_2,A_3\}$ inscribed in $K\setminus r$
such that $A_iA_j\cap r=P_k$, where $i,j,k$ is a permutation of $1,2,3$.
Moreover, if $r$ is a tangent to $K$ then there is one and only one such
triangle inscribed in $K\setminus r$.
\end{enumerate}
\end{theorem}

Actually, one can say even more for $q$ odd, by using an observation of
Korchm\'aros {\cite[Teorema 1]{K1}}. Sometimes this observation is called the
axiom of Pasch for external/internal points.

\begin{proposition}\label{skodd}
Let $K$ be a conic and $r$ be a line of ${\rm PG}(2,q)$, $q$ odd. If  $r$ is not
a tangent and $\{P_1,P_2,P_3\}$ contains either three or exactly one external
point then there are exactly two triangles $\{A_1,A_2,A_3\}$ inscribed in
$K\setminus r$ such that $A_iA_j\cap r=P_k$, where $i,j,k$ is a permutation of
$1,2,3$. In the other cases, for example, when the three points are internal,
there is no $\{A_1,A_2,A_3\}$ with this property.
\end{proposition}

The next result is useful when we wish to determine the intersection of
a conic and a Baer subplane.

\begin{proposition}\label{sk2}
Let $K$ be a conic in $B={\rm PG}(2,q)$, a Baer subplane of ${\rm PG}(2,q^2)$.
Let $r$ be line in $B$. Extend $K$ and $r$ to $K'$ and $r'$ in the larger plane
by using the same equation. Then if $r$ is a tangent, then so is $r'$, otherwise
$r'$ is a secant of $K'$. In other words $K'$ is a parabola if the original
conic $K$ was a parabola, where $r$ and $r'$ are the line at infinity, and it is
a hyperbola if it is not.
\end{proposition}

The difference in extending a hyperbola and an ellipse is that the infinite
points in ${\rm PG}(2,q^2)$ belong to the Baer subplane or not. This
observation can be used to determine the intersection of a Baer subplane and
a conic. The only thing one needs is that five points determine a conic
uniquely.

\begin{corollary}\label{skcor}
Let $B$ be a Baer subplane and $K$ a conic in ${\rm PG}(2,q^2)$. Then either
$|B\cap K|\leq 4$ or $B\cap K$ is a conic of $B$.
\end{corollary}

\section{Consequences for the number of collinear points}

Let $D=\{(m)\mid m\in \GF(q)\cup\{\infty\}\}=\{(x:y:0)\mid x,y\in\GF(q)\}$ be
the standard derivation set we used to define the Hall plane in Section \ref{hall}.
In this section we look at the case that the line
at infinity, $\ell_\infty$ is not a tangent.

We first consider the case that $q$ is even (and at least $4$).

\begin{proposition}\label{hypeven}
If $K$ is a hyperbola either having two points in $D$, or two conjugate points
outside $D$, then $K$ is defined over a subplane (containing $D$), hence in the
Hall plane it has $q-1$ or $q+1$ collinear points, and the remaining lines
of the Hall plane intersect $K$ in at most two points.
\end{proposition}

\begin{proof}
Let $P_1, P_2,$ and $P_3$ be any three points on $\ell_\infty$, and $A_1, A_2$
and $A_3$ be the affine points on $K$ described in Theorem~\ref{sk1}. Notice
that, $\{P_1,P_2,P_3\} \subseteq  D$ iff all points $A_i$ belong to a subplane
containing $D$. In particular, we can fix three affine points of $K$ contained
in a subplane containing $D$, and they together with the two points at
infinity determine a conic (and this is of course $K$), whose homogeneous part
of degree $2,$ which is determined by the infinite points, can be given
coefficients from $\GF(q)$ and therefore, $K$ intersects this subplane in a
subconic. If the two infinite points belong to $D$, then we find $q-1$
collinear points, if they are conjugate, we find $q+1$ collinear points in the
Hall plane.
\end{proof}

\begin{theorem}\label{resteven}
For $q$ even the following hold.
\begin{enumerate}[(a)]
\item If $K$ is a hyperbola having two non-conjugate points on
$\ell_\infty\setminus D$, or if $K$ is an ellipse, then every line of the Hall
plane intersects the affine part of $K$ in at most $4$ points and the number of
collinear triples is $\binom {q+1}3$.
\item If $K$ is a hyperbola having one point in $D$, then every line of the Hall
plane intersects the affine part of $K$ in at most $3$ points and the number of collinear
triples is $\binom q3$.
\end{enumerate}
\end{theorem}

\begin{proof}
By Corollary \ref{skcor} the lines intersect $K$ in at most $4$ points,
and if $K$ has a point in $D$ then at most $3$, since
in this case $K$ does not intersect a Baer subplane containing $D$ in a conic. If $K$ has one point in $D$, then from the remaining $q$ points we get $\binom q3$
triples, and by  {Theorem}~\ref{sk1} the same number of triples in the intersection of $K$ with
a subplane containing $D$, otherwise $K$ has no points in $D$ and we find
$\binom{q+1}3$ such triples.
\end{proof}


Next we consider the case that $q$ is odd. In this case we have the following possibilities:

\begin{enumerate}[(1)]
\item All points of $D\setminus K$ are internal. Now we get from  {Proposition}~\ref{skodd}
that there are no collinear
triples, so we get an inherited arc.
\item $D$ contains $s>0$ external points. In this case we have roughly, but
definitely at least
$\binom s3+s\binom{q-1-s}2$ collinear triples, so certainly $K$ does not
give rise to an arc.
\end{enumerate}

In the next section we will investigate the possible values of $s$.

\section{External points in the derivation set}

We consider the case that $q$ is odd and want to
determine the number of external/internal
points of the conic in the derivation set.

The line at infinity is the line with equation 
${Z}=0$. $D$ is the standard derivation set defined above. The conic $K$ is
given by {$Q(X,Y,Z)=X^2+aXY+bY^2+Z L(X,Y,Z)=0$,} or just by
{$X^2+aXY+bY^2+L(X,Y)$,} where of course {$L(X,Y)=L(X,Y,1)$.} Note that $K$ is
an ellipse if {$f=X^2+aXY+bY^2$} is irreducible over $\GF(q^2)$, a parabola if
$f$ is a square, and a hyperbola if $f$ factors into different linear factors.
For convenience we take $L$ so that the point $(1:0:0)$ is external, and now the
infinite point $(u):=(1:u:0)$ is external/internal when
$1+au+bu^2=$\,\raisebox{-.6pt}{$\Box$}~  or\raisebox{-.6pt}{$~\not\!\Box$}.

Remark: it is an exercise to show that if $P_1$ and $P_2$ are two (external)
points on the same tangent, then either $Q(P_i)$ is a square for both points, or
a non-square. As a consequence $Q(P)$ either is a square for all external points
$P$, or a non-square. This is essentially Theorem 8.17 in \cite{Hir}.

To count the number of external/internal points in $D$, we therefore have to find the number
of (affine) rational points
(so $u,w\in\GF(q)$) on the curve $\mathcal C$ with equation
\[
(1+au+bu^2)(1+\bar a u+\bar b u^2)-w^2=p(u)-w^2=0.
\]
This curve is absolutely irreducible unless the polynomial $p$ is a square. One possibility for this
is that $1+au+bu^2$ is a square, in which case the conic is a parabola. The line at infinity is a tangent
in this case, so we have:

\begin{proposition}\label{para}
If $K$ is a parabola then all points in $D$ different from the infinite point of $K$
are external.
\end{proposition}

The other possibility if $p$ is a square, is that $1+au+bu^2=1+\bar a u+\bar b
u^2$ and now $a,b\in\GF(q)$, so $1+au+bu^2$ factors over $\GF(q^2)$. In this
case the conic has two points at infinity so we have a hyperbola, and we have:

\begin{proposition}\label{hyp0}
If $K$ is a hyperbola and either both infinite points belong to $D$, or they
are conjugates, $(m)$ and $(\bar m)$, both outside $D$, then either all (other) points of $D$ are external,
or all are internal.
\end{proposition}

If $p$ is not a square, then we first take care of the case that $p$ has a
repeated factor. If $1+au+bu^2=(1-\alpha u)(1-\beta u)$, then $1+\bar a
u+\bar b u^2=(1-\bar\alpha u)(1-\bar\beta u)$ and if now $\alpha=\bar \beta$
then $\beta=\bar\alpha$, so $p$ is a square, and we are back in the {case of a
hyperbola with conjugate infinite points,} while if $\alpha=\bar\alpha$ but
$\beta\ne\bar\beta$ then, $K$ has one point in $D$, namely $(\alpha:1:0)$ and
one outside $D$ namely $(\beta:1:0)$, and we now look for the number of points
on the curve
\[
(1+\alpha u)^2(1+\beta u)(1+\bar\beta u)-w^2,
\]
and this is essentially a conic, possibly with some points at infinity.

\begin{proposition}\label{hyp1}
If $K$ is a hyperbola with exactly one infinite point in $D$, then $D$ contains $(q+1)/2$ external
and $(q-1)/2$ internal points, or the other way around, depending on the quadratic character of $\beta\bar\beta$
in $\GF(q)$.
\end{proposition}

So in the case of an ellipse, or a hyperbola with
two non-conjugate points outside $D$ we have no repeated factor, and now by \cite[Example 5.59]{HKT},
$\mathcal{C}$ has genus $g=1$.
Let $R_q$ denote the number of points $P\in \mathcal{C}$ that lie in $\PG(2,q)$. On the one hand, \cite[Theorem 9.57(i)]{HKT} implies
\[|R_q-(q+1)|\leq 2\sqrt{q}+2.\]
On the other hand, $\mathcal{C}$ has a unique point at infinity and all $\GF(q)$-rational affine points $\mathcal{C}$ have the form $(u,\pm w)$ with $w\neq 0$. That is, for $(R_q-1)/2$ values $u\in \GF(q)$, $p(u)$ is a square.
We get:

\begin{proposition}\label{ellhyp}
If $K$ is an ellipse, or a hyperbola with two non-conjugate infinite points
outside $D$, then the number of internal
(external) points on $D$ is at least $q/2-1-\sqrt{q}$ (at most $q/2+1+\sqrt{q}$).
\end{proposition}

\section{Inherited parabolas}

The complete solution to the problem of inherited parabolas was given in a sequence of papers. The story began with the results of Korchm\'aros {\cite[Theorem 1 and 2]{K2}}.

\begin{theorem}
Let $K$ be a parabola in ${\rm PG}(2,q)$ where $q$ is odd. If $K$ is an arc in a
translation plane having the same translation group as the Desarguesian plane,
then the plane must be the Desarguesian one. For $q$ even, there is a parabola
which remains an arc in the Hall plane obtained by derivation.
\end{theorem}

In the case $q$ odd more information is given about parabolas as subsets of the
Hall plane in the paper \cite{SzTJG}. Namely, it is shown that they are sets
having an internal nucleus set that is much larger than a subset of the
Desarguesian plane can have ($P\in S$ is an {\em internal nucleus} if every line
through $P$ contains at most one other point of $S$ \cite{W}). This happens in
the case when the infinite point of the parabola belongs to the derivation set.

If the infinite point of $K$ is not in $D$, then we can use Theorem \ref{sk1},
which gives that for any $\{P_1,P_2,P_3\}\subseteq D$ there are $A_1,A_2,A_3\in
K$ that are collinear in Hall($q^2$). By Proposition \ref{sk2} and Corollary
\ref{skcor} it also follows that every new line intersects $K$ in at most 4
points.

\begin{lemma}
Let $q$ be odd, and let $K$ be a parabola whose infinite point does not
belong to $D$. Then every line of ${\rm Hall}(q^2)$ meets $K$ in at most four points.
\end{lemma}
\begin{proof}
Consider a new line of the affine Hall plane $\Hall(q^2)$. This Baer subplane cannot meet $K$ {in} a subconic, because the infinite point of $K$ does not belong to $D$. Five points in a Baer subplane determine a subconic, hence the Baer subplane can meet $K$ in at most four points.
\end{proof}


Moreover the number of collinear triples is $\binom{q+1}3$. Counting collinear
triples in the Hall plane we get $a_3+4a_4=\binom{q+1}3$, where $a_i$ denotes
the number of lines meeting $K$ in $i$ points. We prove below that the number of
lines in the Hall plane interesecting $K$ in exactly $3$ points does not depend
on the choice of $K$.

Let $K'$ be another parabola with $D \cap K' \ne \emptyset$. There is a
projectivity $\varphi$ that maps $K'$ to $K$ and the infinite point of $K'$ to
the infinite point of $K$. Then $\varphi$ maps $\ell_\infty$ to itself and $D$
to another Baer subline, say $r$. The Baer subplanes containing $D$ are mapped
to the Baer subplanes containing $r$. It is enough to show that there is a
projectivity $\psi$ which fixes $K$ and maps $r$ to $D$ because then the product
$\psi \varphi$ will map the $3$-secant new lines to $K'$ to the $3$-secant new
lines to $K$. Denote by $I$ the infinite point of $K$. Let $G$ be the group of
projectivities fixing $K,$ and $H$ be the stabilizer of $I$ in $G$. The group $G
\cong \PGL(2,q^2),$ which is sharply $3$-transitive on the points of $K$. Thus
$H$ is sharply $2$-transitive on $K \setminus \{I\},$ implying that it acts
doubly transitively on the tangents of $K$ distinct from $\ell_{\infty},$ and
hence also on the points in $\ell_\infty \setminus \{I\}$. When we identify
$\ell_\infty\setminus \{I\}$ with $\GF(q^2)$, then $H$ acts as the set of maps
$z\mapsto az+b$, $a\in GF(q^2)^*$, $b\in \GF(q^2)$ and Baer subplanes not
containing $I$ are circles $(z-c)(\bar z-\bar c)=r$ so we see that $H$ contains
a projectivity $\psi$ that maps the first pair to the second.  Clearly, $\psi$
will map $r$ to $D,$ and by this we showed that the number of lines in the Hall
plane interesecting $K$ in exactly $3$ points does not depend on the choice of
$K$.

 \begin{lemma}\label{l:3-secant}
Let $q$ be odd, and $P_1, \, P_2$ and $P_3$ be three affine
points on a new line $\ell$ of the Hall plane. Then there are
exactly $3(q-1)$ parabolas whose infinite points are not in $D$ and which intersect $\ell$
in exactly $P_1, \, P_2$ and $P_3$.
\end{lemma}
\begin{proof}
Let us write $P_i=(a_i,b_i)$ for $i=1,2,3$. The translation $(x,y) \mapsto (x,y)
- (a_1,b_1)$ maps $\ell$ to a new line through the point $(0,0),$ and therefore,
the affine points of the latter new line form the set $\{(\lambda x,\lambda y)
\mid x,y \in \GF(q)\}$ for some $\lambda\in \GF(q^2)^*$. There exists a
non-singular matrix $A$ with entries in $\GF(q)$ such that
$(a_2-a_1,b_2-b_1)A=\lambda(-1,0)$ and $(a_3-a_1,b_3-b_1)A=\lambda(0,-1)$. Let
$\varphi$ be  the automorphism of $\AG(q^2)$ defined by $\varphi : (x,y) \mapsto
\lambda^{-1}(x-a_1,y-b_1)A$. This extends naturally to a projectivity of
$\PG(2,q),$ which fixes $\ell_\infty$ setwise, and maps $D$ to itself. The image
$\varphi(\ell)$ is the new line for which
\[
\varphi(\ell) \, \setminus \, \ell_\infty =\{(x,y) \mid x,y \in \GF(q)\}.
\]

We are done if we show that there are exactly $3(q-1)$
parabolas whose infinite points are not in $D$ and which intersect $\varphi(\ell)$ in
exactly the points $(0,0), \, (-1,0)$ and $(0,-1)$.

For $u \in \GF(q^2) \setminus  \GF(q),$
denote by $K_u$ be the unique parabola that
contains the points $(0,0), \, (-1,0)$ and $(0,-1)$ and the
infinite point $(-u : 1 : 0)$.
Then $K_u$ has affine equation
\[
f(X,Y)=(X+uY)^2+ X+u^2 Y=0.
\]

We find next all $\GF(q)$-rational points of
$K_u$. If $P=(a,b)$ is such a point, then we compute
\begin{eqnarray*}
f(a,b)-\overline{f(a,b)} &=& (u-\bar{u})b(2a+(u+\bar{u})(b+1))=0, \\
-\bar{u}^2 f(a,b)+u^2 \overline{f(a,b)} &=&
(u - \bar{u})a( (u+\bar{u})(a+1)+2u \bar{u} b)=0.
\end{eqnarray*}
Since $u-\bar{u} \ne 0,$ these show that $a=0$ or $b=0$ (and in this
case $(a,b) \in \{(0,0),(-1,0),(0,-1)\}$), or
$(a,b)$ can be obtained as the unique solution of a system of linear equations, which then
yields
\[ P=\Big(\, \frac{(u+\bar{u})(2u\bar{u}-u-\bar{u})}{(u-\bar{u})^2},
\, \frac{(u+\bar{u})(2-u-\bar{u})}{(u-\bar{u})^2} \, \Big). \]
It is clear that $P$ is $\GF(q)$-rational, and we leave for the reader to check that it
lies on $K_u$.

We conclude that $|\varphi(\ell) \cap K_u| \in \{3,4\},$ and $|\varphi(\ell)
\cap K_u|=3$ iff $P$ is equal to one of the points $(0,0), \, (-1,0)$ and
$(0,-1)$. A quick computation gives that this occurs iff $u+\bar{u} \in \{0,2\}$
or $1/u+1/\bar{u}=2$. It can be easily checked that $u$ satisfies one of the
latter conditions iff $u \in \GF(q)^* \omega$ or $u \in \GF(q)^*\omega+1$ or $u
\in (\GF(q)^*\omega +1)^{-1},$ where $\omega \in \GF(q^2)$ is any element such
that $\omega^2$ is a non-square in $\GF(q)$. This implies that  there are
$3(q-1)$ parabolas $K_u$ intersecting $\varphi(\ell)$ in exactly the points
$(0,0), \, (-1,0)$ and $(0,-1),$ and this completes the proof of the lemma.
\end{proof}

\begin{theorem}
Let $D$ be a derivation set on $\ell_\infty$ of $\AG(2,q^2)$ with $q$ odd. Let $K$ be a parabola whose infinite point does not belong to $D$. Then there are $a_3=(q^2-1)/2$ and $a_4=(q-3)(q^2-1)/24$ lines of $\Hall(q^2)$ meeting $K$ in $3$ or $4$ points, respectively. 
\end{theorem}
\begin{proof} 
Let $U$ be the set of parabolas whose infinite point does not belong to $D$. Any element of $U$ has a uniquely defined equation of the form 
\[Y=\alpha (X-uY)^2+\beta(X-uY)+\gamma,\] 
with $u\in \GF(q^2)\setminus \GF(q)$, $\alpha \in \GF(q^2)^*$, $\beta,\gamma \in \GF(q^2)$. Hence, \[|U|=(q-1)(q^2-1)q^5.\] 
We showed above that for any $K\in U$, the number of $3$-secant new lines is a constant $a_3$. 

Let $V$ be the set of new lines; $|V|=(q+1)q^2$. For the set 
\[W=\{ (K,B,P_1,P_2,P_3) \mid K\in U, B\in V, K\cap B=\{P_1,P_2,P_3\}\},\]
one has
\[|W|=6|U|a_3=|V|q^2(q^2-1)(q^2-q)\cdot 3(q-1)\]
by Lemma \ref{l:3-secant}. The value for $a_4$ follows from $a_3+4a_4=\binom{q+1}{3}$. 
\end{proof}


The case when $q$ is even is more interesting. The four cases are treated by O'Keefe, Pascasio \cite{O'KP}, O'Keefe, Pascasio, and Penttila \cite{O'KPP} and Glynn, Steinke \cite{GS}.

\begin{theorem}[\cite{O'KP}, \cite{O'KPP}, \cite{GS}] Let $D$ be a derivation set on
$\ell_\infty$ of ${\rm AG}(2,q^2)$, with $q\ge 4$ even,
and $K$ a parabola with infinite point $I$ and nucleus $N$.
\begin{enumerate}[(i)]
\item If $I\in D$ and $N\in D$, then $K$ is not an arc in the Hall plane.
\item If $I\not\in D$ and $N\in D$,
then $K$ is a translation
$q^2$-arc in the derived plane and it can be extended to a hyperoval. Any
two hyperovals of the Hall plane arising from this construction are equivalent
under the automorphism group of the Hall plane.
\item If $I\in D$ and $N\not\in D$, then $K$ is a translation
$q^2$-arc in the derived plane and it can be extended to a hyperoval. Any
two hyperovals of the Hall plane arising from this construction are equivalent
under the automorphism group of the Hall plane.
The two cases give inequivalent hyperovals in the Hall plane.
\item If $I\not\in D$ and $N\not\in D$, then $K\cup\{I\}$
is a translation oval if and only if $q$ is a square, and $I$ and $N$ are conjugate
with respect to $D$.
\end{enumerate}
\end{theorem}

Also in the case $q$ even, we know something about the combinatorial structure
of $K$ in $\Hall(q^2)$ if $I,N\in D$. In this case we may assume that the parabola has
equation $K:Y=X^2$ and $D$ is the standard derivation set. Points of $K$
whose coordinates are in $\GF(q)$ are collinear in $\Hall(q^2)$ and
the same is true for points whose first coordinate is in an additive coset of
$\GF(q)$. So the points of $K$ are on $q$ parallel lines. Other triples
are not collinear.

In the general Glynn--Steinke case $I,N\not\in D$, we can show that each
line meets $K$ in $0$, $1$, $2$ or $4$ points.

\begin{lemma}\label{lm:beta}
Let $q$ be a power of $2$ and $\beta \in \GF(q^2)^*$. Let $N_\beta$ be the number of $\GF(q)$-rational roots of
\[f(T)=T^3+\beta\bar{\beta}T + \beta\bar{\beta}(\beta+\bar{\beta}).\]
\begin{enumerate}[(a)]
\item If $q$ is a square then
\[N_\beta = \begin{cases}
3 & \mbox{for $\beta \in \GF(q)$,} \\
1 & \mbox{for $\beta \in \GF(q^2)\setminus\GF(q)$.}
\end{cases}\]
\item If $q$ is not a square then
\[N_\beta = \begin{cases}
3 & \mbox{if $\beta$ is a cube in $\GF(q^2)^*$,} \\
0 & \mbox{otherwise.}
\end{cases}\]
\end{enumerate}

\end{lemma}
\begin{proof} 
If $\beta=\bar{\beta}$ then the roots of $f(T)$ are $0,\beta,\beta$, in accordance with (a) and (b). For the remaining part, we assume $\beta\neq \bar{\beta}$. Let $\varepsilon, d$ be elements of $\overline{\GF(q)}$ such that $\varepsilon^2+\varepsilon+1=0$ and $d^3=\beta$. Then, the three different roots of $f(T)$ are
\begin{align*}
t_1 &= d^{q+1}(d+d^q),\\
t_2 &= d^{q+1}(\varepsilon^2 d+\varepsilon d^q),\\
t_3 &= d^{q+1}(\varepsilon d+ \varepsilon^2 d^q).
\end{align*}
($\beta\neq \bar{\beta}$ implies $t_i\neq t_j$ for $i\neq j$.)

{Assume that $q$ is not a square. Then
$\varepsilon^q=\varepsilon^2,$ and thus}
if $\beta$ is a cube in $\GF(q^2)$, then $d\in \GF(q^2),$ and $t_1,t_2,t_3\in \GF(q)$. If $d\not \in \GF(q^2)$ then the three cubic roots of $\beta$ are $d, d^{q^2}, d^{q^4}$, and we have $ d^{q^2}=\varepsilon d$. This implies $t_1^q=t_2$, $t_2^q=t_3$ and $t_3^q=t_1$, showing that no root of $f(T)$ lies in $\GF(q)$. This proves (b).

Now, let $q$ be a square. Then $\varepsilon^q=\varepsilon,$ and we obtain that $t_1^q=t_1,$ $t_2^q=t_3$ and $t_3^q=t_2$ when $\beta$ is a cube in $\GF(q^2)$,  and $t_3^q=t_3,$ $t_1^q=t_2$ and $t_2^q=t_1$
when $\beta$ is not a cube. In either case $f(T)$ has one root in $\GF(q),$ as claimed in (a).
\end{proof}

\begin{theorem}
Let $D$ be a derivation set on
$\ell_\infty$ of ${\rm AG}(2,q^2)$, with $q\ge 4$ even, and $K$ a parabola with infinite point $I$ and nucleus $N$. Assume that $I\not\in D$ and $N\not\in D$. Then the following holds:
\begin{enumerate}[(i)]
\item Each line of the Hall plane intersects $K$ in $0,1,2$ or $4$ points.
\item If $I$ and $N$ are conjugate with respect to $D$, and $q$ is not a square, then each point $P\in K$ is contained in $(q+1)/3$ $4$-secant new lines and $2(q+1)/3$ $1$-secant new lines. In particular, the Hall plane has no $2$-secant new lines.
\end{enumerate}
\end{theorem}
\begin{proof} 
Let $I$ and $N$ be the points $(u:1:0)$ and $(v:1:0)$ of the line at infinity; $u,v\in \GF(q^2)\setminus\GF(q)$. Then, the homogeneous equation of $K$ has the form
\[X^2+u^2Y^2+\beta_0 Z (X+vY)+\beta_1 Z^2=0,\]
where $\beta_0 \in \GF(q^2)^*$ and $\beta_1\in \GF(q^2)$. Let $\ell$ be a new line of the Hall plane and assume $K\cap \ell \neq \emptyset$. W.l.o.g. we can assume that $(0,0) \in K\cap \ell$. Then $\beta_1=0$ and
\[\ell \, {\setminus \, \ell_\infty}
=\{(\lambda x, \lambda y,1) \mid x,y \in \GF(q)\}\]
for some $\lambda \in \GF(q^2)^*$. In order to compute $K\cap \ell$, we substitute $X=\lambda x$, $Y=\lambda y$, $Z=1$ in the equation of $K$. We obtain
\[C:x^2+u^2y^2+\beta(x+vy)=0,\]
where $\beta=\beta_0/\lambda \in \GF(q^2)^*$. The $\GF(q)$-rational points of $C$ are contained in $C\cap \bar{C}$.

Assume $\beta=\bar{\beta}\in \GF(q)^*$. Then $C+\bar{C}: (u^2+\bar{u}^2)y^2+\beta(v+\bar{v})y=0$, giving two $\GF(q)$-rational roots $y_1=0$ and $y_2=\frac{\beta(v+\bar{v})}{u^2+\bar{u}^2}=\gamma \in \GF(q)^*$. For $y_1=0$, we get $x_1=0$ or $x_2=\beta$, two rational points. For $y_2=\gamma$, we get two different roots $x_3,x_4$ of $x^2+\beta x+u^2\gamma^2+\beta v \gamma$. This means $2$ or $4$ $\GF(q)$-rational points on $C$.

Assume now $\beta \neq \bar{\beta}$ and compute the resultant
\[R_{C,\bar{C}}(y)=(u+\bar{u})^4 y^4 + [(u+\bar{u})^2\beta\bar{\beta} +  (u+\bar{v})^2\bar{\beta}^2 + (v+\bar{u})^2\beta^2]y^2 + \beta\bar{\beta} (\beta+\bar{\beta}) (v+\bar{v}) y.\]
As the derivative is a nonzero constant, this resultant has four different roots. Clearly, if three of them sit in $\GF(q)$ then so does the fourth. If $y=\gamma$ is a rational root of the resultant, then
\[x^2+\beta x+ u^2\gamma^2+\beta v \gamma, \quad x^2+\bar{\beta} x+ \bar{u}^2\gamma^2+\bar{\beta} \bar{v} \gamma\]
have a unique rational common root, giving rise to a unique $\GF(q)$-rational point of $C$.
{In particular, $\ell$ intersects $K$
in $0, 1, 2$ or $4$ points, and this together with the previous paragraph
shows that (i) holds.}

We turn now to the statement in (ii), and assume that $q$ is not a square and $I$ and $N$ are conjugate w.r.t.\ $D$. This means $u=\bar{v}$ and the resultant $R_{C,\bar{C}}(y)$ becomes
\[r(T)=T^4+\beta\bar{\beta}T^2 + \beta\bar{\beta}(\beta+\bar{\beta})T,\]
where $T=(u+\bar{u})y$. By Lemma \ref{lm:beta}, $r(T)$ has $1$ or $4$ $\GF(q)$-rational roots, depending whether $\beta$ is a cube or not in $\GF(q^2)^*$. Moreover, these roots are different for $\beta\neq \bar{\beta}$,
and hence $\ell$ is a $1$- or a $4$-secant of $K$ depending whether $\beta$ is a cube or not.
If $\beta=\bar{\beta},$ then a straightforward calculation shows,
that in this case, the four points of $C\cap \bar{C}$ are
\[(0,0), \quad (0,\beta), \quad \left(\frac{u\beta}{u+\bar{u}} + \varepsilon \beta, \frac{\beta}{u+\bar{u}} \right), \quad \left(\frac{u\beta}{u+\bar{u}} + \varepsilon^2 \beta, \frac{\beta}{u+\bar{u}} \right),\]
where $\varepsilon^2+\varepsilon+1=0$. The last two points are $\GF(q)$-rational iff $\varepsilon+\varepsilon^q=1$, which holds iff $q$ is not a square.
The multiplicative group $\GF(q^2)^*$ is a cyclic group of order $q^2-1,$ let $\mathcal{K}$ and $\mathcal{L}$ be its unique subgroups of order $q-1$ and $(q^2-1)/3,$ resp.
As $q$ is not a square, $(q-1)$ divides $(q^2-1)/3,$ and thus $\mathcal{K} < \mathcal{L}$. Our above discussion shows that the new line $\ell$ is a $1$- or a $4$-secant of $K,$ and that it is a $4$-secant is equivalent to say that $\beta \in \mathcal{L}$.

Recall that, $\beta=\beta_0/\lambda,$ where $\beta_0$ is some
fixed element in $\GF(q^2)^*,$ and $\lambda \in \GF(q^2)^*$
defines the new line $\ell$. The $q+1$ new lines through the affine point $(0,0)$ can be listed by letting $\lambda$ run over any complete set of coset representatives of $\mathcal{K}$ in
$\GF(q^2)^*$.  Now, denoting by $\Lambda$ such a set of coset representatives, the number of $4$-secants through
$(0,0)$ is equal to
\[
|\{\lambda \in \Lambda : \beta_0/\lambda \in \mathcal{L}\}|=
|\Lambda \cap \mathcal{L}\beta_0|.
\]

Consider the canonical projection $\eta : \GF(q^2)^\ast \to \GF(q^2)^\ast/\mathcal{K}$. It follows that
$\eta(\Lambda)=GF(q^2)^\ast/\mathcal{K},$ and $\eta$ induces a bijection from $\Lambda \cap \mathcal{L}\beta_0$ to $\eta(\Lambda \cap \mathcal{L}\beta_0)=GF(q^2)^\ast/\mathcal{K} \cap \mathcal{L}/\mathcal{K}\, (\mathcal{K}\beta_0)=\mathcal{L}/\mathcal{K} \, (\mathcal{K}\beta_0)$ (here $\mathcal{K} \beta_0$ is regarded as an element in $\GF(q^2)^\ast/\mathcal{K}$).
This gives $|\Lambda \cap \mathcal{L}\beta_0|= |\mathcal{L}/\mathcal{K}|=(q+1)/3,$ and (ii) follows.
\end{proof}

Part (i) of the last proposition also follows from the proof of Glynn--Steinke, see \cite[Section 4]{GS}.

Part (ii) implies that if $I,N$ are conjugate w.r.t.\ $D$ and $q$ is not a square, then the number of $i$-secant new lines of the Hall plane is $a_0=\frac{1}{4}q^2(q+1)$, $a_1=\frac{2}{3}q^2(q+1)$ and $a_4=\frac{1}{12}q^2(q+1)$ for $i=0,1,4$.

\section{Inherited hyperbolas}

A surprising phenomenon occurs in this case. When the infinite points of
a hyperbola belong to the derivation set, then it is possible that although
the affine points of the hyperbola form an inherited arc, this arc is
complete. This was pointed out in \cite{SzTJG} and the possible
configurations were fully described by O'Keefe and Pascasio.
Note that in Galois planes there are no complete $(q-1)$-arcs by the
theorems of Segre mentioned in the introduction.

\begin{theorem}[O'Keefe-Pascasio, \cite{O'KP}]
Suppose that the line at infinity is a secant of a hyperbola $K$ whose
infinite points belong to the derivation set $D$. Assume that $q>3$ is odd and
$D$ is the standard derivation set.
Then either we have $K$ equivalent to the hyperbola $XY=1$ which does not
give an inherited arc, or to the hyperbola $XY=-d$ with $d$ a non-square
in $\GF(q^2)$ and we get a complete $(q^2-1)$-arc in ${\rm Hall}(q^2)$.
For $q>2$ even, and $D$ standard, $K$ is equivalent to the hyperbola $XY=1$
and does not give an inherited arc.
\end{theorem}

Note that the odd case of the above theorem essentially is (one part of)
Proposition \ref{hyp0}.

O'Keefe and Pascasio \cite{O'KP} give a complete description of the
resulting configurations in the Hall plane for $q=3$.

The next case we consider is that the line at infinity is a secant of the
hyperbola $K$ with two conjugate infinite points outside the derivation set.

\begin{theorem}
Suppose that the line at infinity is a secant of a hyperbola $K$ whose infinite points are
conjugate, so outside of the (standard) derivation set $D$. Assume that $q>3$ is
odd. Then either all points of $D$ are internal, and $K$ (together with the two
infinite points) is an inherited oval in the Hall plane, or all points of $D$
are external and now we find (two) lines containing $q+1$ points of $K$.
\end{theorem}

The first case is just Proposition \ref{hyp0} together with Proposition \ref{skodd}.
If all points are external, then again by Proposition \ref{skodd}, for every triple
of points in $D$, there are two corresponding triangles in $K$, and these
together form two ellipses in two Baer subplanes on $D$.

Remark: if $q$ is even, and the two infinite points of $K$ are conjugate, then
we find exactly one line in the Hall plane with $q+1$ points of $K$ as a
consequence of Proposition \ref{hypeven}.

The third case to consider is that the line at infinity is a secant of the
hyperbola $K$ with one point in the derivation set, and one outside.

The following proposition makes more precise what we already mentioned in the
section about consequences of the theorem by Segre and Korchm\'aros, together
with {Proposition}~\ref{hyp1}.

\begin{theorem}\label{hsz} 
If $q>5$ is odd and $K$ is a hyperbola with one point in the derivation set, and
one point outside of it, then the affine part of $K$ does not give an arc in the
Hall plane, moreover, lines of the Hall plane intersect it in at most three
points. If $s=\frac12(q\pm 1)$ denotes the number of external points on our
derivation set $D$, then the total number of collinear triples in the Hall plane
is $2s\binom {q-s}{2}+ 2\binom s3$, which for $q$ large enough is roughly
$7q^3/48$, so small.
\end{theorem}

The final case to consider is where both infinite points of the hyperbola $K$
are outside the derivation set $D$ and are not conjugate. From Proposition
\ref{ellhyp} we know that the number of internal/external points in $D$ is at
most $q/2-1-\sqrt{q}$.

If $s$ denotes the number of external points in our derivation set $D$, then the
total number of collinear triples in the Hall plane is $2s\binom
{q+1-s}2+2\binom s3$, roughly $7q^3/48$, using the above bound on $s$. Note also
that we do not have collinear sets of size 5 or more, since 5 points of our Baer
subplanes extend to a conic with points at infinity. We will return to this case in
section \ref{sec:ellhyp} where the case $q$ even is studied in more detail.

\section{Inherited ellipses for $q$ odd}

The last case to consider is an ellipse $K$ in the affine plane $\AG(2,q^2)$, so
a conic without points on the line at infinity. Let $q$ be odd. Then, on the
line at infinity $\ell_\infty$ we have $(q^2+1)/2$ external and $(q^2+1)/2$
internal points. If $D$ is a Baer subline of $\ell_\infty$ then $K$ is again an
oval in the derived plane if and only if the derivation set $D$ is disjoint from
the set of external points (on $\ell_\infty$), as a consequence of Proposition
\ref{skodd}. In Proposition \ref{ellhyp} we have seen that this is impossible
for $q>7$. The following combinatorial proof works for all $q$.

\begin{theorem}
Let $q$ be odd, $K$ an ellipse in $\AG(2,q^2)$, then $K$ does not remain an
oval in ${\rm Hall}(q^2)$.
\end{theorem}

\begin{proof}
Here we essentially just count. Consider the line $\ell_\infty$ together with the
partition $E\cup I$ into external and internal points. The subgroup of ${\rm
PGL}(2,q^2)$ stabilizing this partition has order $2(q^2+1)$, there is a
dihedral group of order $q^2+1$ fixing the set $E$ and an extra factor $2$
because we may interchange $E$ and $I$. Now how does this group act on the set
of Baer-sublines, or better, how large are the orbits? The stabilizer of a
Baer-subline has order $(q+1)q(q-1)$, and the greatest common divisor of
$(q^2+1)$ and $(q+1)q(q-1)$ is $2$, this means that if we find a Baer-subline
contained in $E$, we find $(q^2+1)/2$, and an additional set of this size in
$I$. Now let us count the number $N$ of triples $(P_e,P_i,B)$, of an external
point $P_e$, an internal point $P_i$ and a Baer-subline $B$ containing them.
Since every pair of points is contained in $(q^2-1)/(q-1)=(q+1)$ Baer-sublines,
we find $N=\frac14(q^2+1)^2(q+1)$. Now we count in the other way, the total
number of Baer-sublines is $(q^2+1)q^2(q^2-1)/((q+1)q(q-1))=(q^2+1)q$, but if we
assume that there is a Baer-subline contained in $E$, then at least $(q^2+1)$ of
them do not  contribute to our counts, so we find $N\le
\frac14(q^2+1)(q-1)(q+1)^2<\frac14(q^2+1)^2(q+1)$, contradiction.
\end{proof}

\begin{remark}
\begin{enumerate}[(i)]
\item If $s$ denotes the number of external points in our derivation set $D$, then as
in the hyperbola case the total number of collinear triples in the Hall plane is
$2s\binom {q+1-s}2+ 2\binom s3$, roughly $7q^3/48$, using the bound on $s$ from
Proposition \ref{ellhyp}. 
\item Also, we do not have collinear sets of size 5 or more, since 5 points in
of our Baer-subplanes extend to a conic with points at infinity.
\end{enumerate}
\end{remark}

We finish this section with an open problem concerning the exact number of $3$-secant new lines. 

\begin{question}
Let $q$ be odd, $K$ an ellipse or a hyperbola with non-conjugate infinite points in $\AG(2,q^2)$. Let $s$ denote the number of external points of $K$ in the derivation set $D$. Find a formula for the number $a_3$ of $3$-secant new lines in terms of $q$ and $s$. 
\end{question} 

In the last section, we answer this question for the even $q$ case. 

\section{Inherited ellipses and hyperbolas for $q$ even} \label{sec:ellhyp}

In Theorem \ref{resteven}, we showed that if $q$ is even and $K$ is either
an ellipse or a hyperbola having two non-conjugate infinite points, then in the
Hall plane $K$ has $\binom{q+1}{3}$ collinear triples. In this section, we give
an explicit formula for the number $a_3$ of $3$-secant new lines. By $a_3+4a_4=\binom{q+1}3$ 
this also determines the number $a_4$ of $4$-secant new lines. 

\begin{theorem} \label{th:ellhyp}
Let $q$ be even, and $D$ a derivation set on $\ell_\infty$ of $\AG(2,q^2)$. Let $K$ be either an ellipse or a hyperbola such that the infinite points of $K$ are non-conjugate and none of them is contained in $D$. Then, the number of $3$-secant new lines is $a_3=q(q-1)/2$.
\end{theorem}

\begin{lemma} \label{lm:qform_c}
Let $u_1,u_2\in \GF(q^4)$ be the roots of the quadratic polynomial $f(X)=X^2+\beta X+\gamma \in \GF(q^2)[X]$ and assume $u_i \not\in \{ u_i^q, u_j, u_j^q\}$, where $\{i,j\}=\{1,2\}$. Then there is a $\GF(q)$-rational map $z\mapsto \frac{az+b}{cz+d}$ which brings $f(X)$ to the form $X^2+X+w$ with some $w\in \GF(q^2)\setminus \GF(q)$. 
\end{lemma}
\begin{proof}
The fact $u_i\neq u_j$ implies $\beta\neq 0$. Straightforward calculation shows
\[f(u(z))=\frac{(a^2+\beta ac+\gamma c^2)z^2 + \beta(cb+ad)z +b^2+\beta db+\gamma d^2}{(cz+d)^2}.\]

Assume first that $\beta \in \GF(q^2)\setminus \GF(q)$. If $\gamma = t\beta$ with $t\in \GF(q)$ then $u(z)=t/z$ brings $f(Z)$ to the form $X^2+X+t/\beta$. If $\gamma/\beta \not\in \GF(q)$, then $\beta,\gamma$ forms a $\GF(q)$-basis of $\GF(q^2)$ and there are unique elements $t_1,t_2\in\GF(q)$, $t_2\neq 0$, such that 
\[1=t_1\beta + t_2\gamma.\]
Define $a=b=1$, $c=\sqrt{t_2}$, $d=1+\sqrt{t_2}$. Then 
\begin{align}
a^2+\beta ac+\gamma c^2 &= (t_1+\sqrt{t_2})\beta, \label{eq:f1}\\
\beta(cb+ad)&=\beta,\label{eq:f2}\\
b^2+\beta bd+\gamma d^2 &= \beta(1+t_1+\sqrt{t_2})+\gamma.\label{eq:f3}
\end{align}
By \eqref{eq:f1}, $t_1+\sqrt{t_2}=0$ implies that $c/a=\sqrt{t_2}$ is a root of $f(X)$, which contradicts to $u_i\neq u_i^q$. Hence, $u$ brings $f(X)$ to the form $f_0(X)=X^2+\beta_0 X+\gamma_0$, with
\[\beta_0=\frac{1}{t_1+\sqrt{t_2}}\in GF(q)^*.\]
Now, $v(z)=\beta_0 z$ brings $f_0(X)$ to the desired form $X^2+X+w$, where $w\not\in \GF(q)$ follows from $u_i\neq u_j^q$. 
\end{proof}

Lemma \ref{lm:qform_c} implies that $\AG(2,q^2)$ has an affine coordinate frame in which $D=\{ (x,y,0) \mid x,y\in \GF(q)\}$ and the equation of $K$ has the form  
\[K: X^2+XY+cY^2+uX+vY+w,\]
with $c\in \GF(q^2)\setminus \GF(q)$, $u,v,w\in \GF(q^2)$. All dilations (=translations and homotheties) preserve $D$ and the quadratic component $X^2+XY+cY^2$ of $K$. 

We use the notation $t_P$ for the tangent line of $K$ at the point $P\in K$. 

\begin{lemma} \label{lm:3pts_tang}
Let $B$ be a new line. 
\begin{enumerate}[(i)]
\item If $|B\cap K|=3$ then there is a unique $P\in B\cap K$ such that the tangent $t_P$ intersects $D$. 
\item If $t_P$ intersects $D$ for an element $P\in B\cap K$ then $|B\cap K|\leq 3$. 
\end{enumerate}
\end{lemma}
\begin{proof}
Up to dilatations we can assume that $B=\{(x,y)\mid x,y\in \GF(q)\}$, which means that $B\cap K$ consists of the $\GF(q)$-rational points of $K$. Equivalently, $B\cap K = K\cap \bar{K}$, where 
\[\bar{K}: X^2+XY+c^q Y^2+u^q X+v^q Y+w^q\]
is the conjugate of $K$. Counting with multiplicities, $K$ and $\bar{K}$ have $4$ points in common over the algebraic closure of $\GF(q)$. (Cf. B\'ezout's Theorem \cite[Theorem 3.14]{HKT}). If $|B\cap K|=3$ then there is a unique $P\in B\cap K$ such that $K$ and $\bar{K}$ have intersection multiplicity $2$. In particular, $K$ and $\bar{K}$ have a common tangent $t$ at $P$ (see \cite[Proposition 3.6]{HKT}). This means that $t$ is defined over $\GF(q)$ and the infinite point of $t$ is in $D$. This proves (i). 

Conversely, if $t_P\cap D\neq \emptyset$ for some $P\in B\cap K$, then $t_P$ is a common tangent of $K$ and $\bar{K}$. Hence, $K$ and $\bar{K}$ have a common tangent at $P$, which implies an intersection multiplicity at least $2$. Thus, (ii) follows.
\end{proof}


\begin{proof}[Proof of Theorem \ref{th:ellhyp}.]
Fix a point $A=(u,v,0) \in D$. Let $t$ be the tangent from $A$ to $K$, with tangent point $T\in K$. We want to determine the 3-secant new lines through $T$. Lemma \ref{lm:3pts_tang} shows that if $A$ runs through $D$, then this enumerates all 3-secant new lines for $K$. 

Up to dilations, we can have $T=(0,0)$ and let $K$ have equation
\[K:X^2+XY+cY^2+vX+uY.\]
Moreover, any new line through $T$ has the form
\[ B_\lambda=\{(\lambda^{-1}x,\lambda^{-1}y) \mid x,y\in \GF(q) \}\]
for some $\lambda \in\GF(q^2)$. More precisely, since $\lambda$ is given up to a nonzero $\GF(q)$-rational scalar multiple, w.l.o.g. $\lambda=1$ or $\lambda=\lambda_0+c$ with $\lambda_0\in \GF(q)$. By substituting the generic point of $B_\lambda$ in the equation of $K$, we find a $1-1$ correspondence between $B_\lambda \cap K$ and the set of $\GF(q)$-rational points of the conic
\[K_\lambda: X^2+XY+cY^2+v\lambda X+ u\lambda Y.\]

\underline{Case 1: $\lambda=1$.} Then $K_\lambda + \bar{K}_\lambda: Y^2=0$, This implies $Y=0$ and $X^2+vX=0$. The only rational points of $K_\lambda$ are $(0,0)$ and $(v,0)$. Thus, $|B_\lambda \cap K|\leq 2$. 

\underline{Case 2: $\lambda\neq 1$ and $(u,v)=(1,0)$.} As $\lambda=\lambda_0+c$, we have $\lambda+\lambda^q=c+c^q$ and $K_\lambda + \bar{K}_\lambda: (c+c^q)(Y^2+Y)=0$. If $Y=0$ then $X=0$. In order to have $|B_\lambda\cap K|=3$, we need two more rational points, which holds if $Y=1$ and $X^2+X+\lambda_0$ has two distinct roots in $\GF(q)$. This happens if and only if $\mathrm{Tr}_{\GF(q)/\GF(2)}(\lambda_0)=0$. Therefore, we found exactly $q/2$ new lines $B_\lambda$ through $T=(0,0)$ such that $|B_\lambda\cap K|=3$. 

\underline{Case 3: $\lambda\neq 1$ and $v=1$.} Again, $\lambda=\lambda_0+c$ and 
\[K_\lambda+\bar{K}_\lambda: (c+c^q)(Y^2+X+u Y)=0. \]
Substituting $X=Y^2+u Y$ into $K_\lambda$, we have
\[Y^2(Y^2+Y+u^2+u+\lambda_0)=0.\]
If $Y=0$ then $X=0$. If
\[\mathrm{Tr}_{\GF(q)/\GF(2)}(u^2+u+\lambda_0)=\mathrm{Tr}_{\GF(q)/\GF(2)}(\lambda_0)=1\]
then $B_\lambda \cap K=\{(0,0)\}$. If $\lambda_0=u^2+u$ then $Y=0$ or $Y=1$, and $B_\lambda \cap K=\{(0,0), (u+1,1)\}$. Finally, if $\mathrm{Tr}_{\GF(q)/\GF(2)}(\lambda_0)=0$ and $\lambda_0\neq u^2+u$, then $Y^2+Y+u^2+u+\lambda_0$ has two roots in $\GF(q)\setminus \{0,1\}$, giving rise to two rational points of $K_\lambda$, different from $(0,0)$. Hence, we found $q/2-1$ new lines $B_\lambda$ through $T=(0,0)$ such that $|B_\lambda\cap K|=3$. 

Resuming the results, we found $q/2+q(q/2-1)=q(q-1)/2$ new lines which intersect $K$ in exactly $3$ points. 
\end{proof}

\end{document}